\documentclass{article}
\usepackage{makeidx}
\makeindex
\usepackage{amsmath}
\usepackage{amssymb}
\usepackage{amsthm}
\usepackage{fullpage}

\begin{document}
\title{Reaction-diffusion on a time-dependent interval: refining the notion of `critical length'}
\date{}
\author{Jane Allwright}
\maketitle

\newtheorem{theorem}{Theorem}[section]
\newtheorem{lemma}{Lemma}[section]
\newtheorem{proposition}{Proposition}[section]
\newtheorem{corollary}{Corollary}[section]
\theoremstyle{definition}
\newtheorem{remark}{Remark}[section]
\newtheorem{example}{Example}[section]
\newtheorem{definition}{Definition}[section]

{\bf Abstract}
A reaction-diffusion equation is studied in a time-dependent interval whose length varies with time. The reaction term is either linear or of KPP type.
On a fixed interval, it is well-known that if the length is less than a certain critical value then the solution tends to zero.
When the domain length may vary with time, we prove conditions
under which the solution does and does not converge to zero in long time. We show that, even with the length always strictly less than the `critical length', either outcome may occur. Examples are given.
The proof is based on upper and lower estimates for the solution, which are derived in this paper for a general time-dependent interval.

\section{Introduction}
We consider the reaction-diffusion problem:
\begin{equation}\label{eq_psi}
\frac{\partial \psi}{\partial t} = D \frac{\partial^2 \psi}{\partial x^2} +f(\psi)  \qquad \textrm{in } A(t)< x< A(t)+L(t)
\end{equation}
\begin{equation}\label{eq_psi_BC}
\psi(x,t)=0 \qquad \textrm \qquad\textrm{at } x=A(t) \textrm{ and } x=A(t)+L(t)
\end{equation}
where $\psi \geq 0$, $D>0$, $L(t)>0$, and the reaction term $f$ is assumed to be either linear:
\begin{equation}
f(\psi)=f'(0)\psi, \qquad f'(0)>0,
\end{equation}
or a Lipschitz continuous function, differentiable at $0$, and satisfying the following conditions:
\begin{equation}\label{eq_f}
f(0)=f(1)=0, \qquad f'(0)>0,\qquad\frac{f(k)}{k} \textrm{ non-increasing on } k>0.
\end{equation}
Note that these assumptions on $f$ imply that $f(k)\leq f'(0)k$ and so the solution to the linear problem is a supersolution to the nonlinear equation.
This type of nonlinearity is often referred to as being of the KPP type (named after the initials of the authors of \cite{KPP}).

Both the start of the interval, $A(t)$, and the length of the interval, $L(t)$, are prescribed functions of time, and are assumed to be twice continuously differentiable, and $L(t)>0$ for all $t\geq 0$.

It is straightforward to show that, for a fixed domain $0<x<L$, the solution to the linear problem tends exponentially to either zero or infinity, depending on whether $f'(0)-\frac{D\pi^2}{L^2} <0$ or $>0$ respectively. This yields the `critical length'
\begin{equation}
L_{crit}=\pi\sqrt{\frac{D}{f'(0)}}.
\end{equation}
Note that if $f'(0)-\frac{D\pi^2}{L^2}=0$ (i.e. $L=L_{crit}$) then the solution converges to a multiple of the principal eigenfunction, $\sin\left(\frac{\pi x}{L}\right)$. (Namely, it converges to the first term in the Fourier sine series for the initial conditions.)
The purpose of the present paper is to present some refinements of the notion of the `critical length' to cases when the domain length is no longer a constant but is able to vary with time. We first treat the linear growth term, and then also extend the results to the nonlinear KPP term. We find that certain conditions on $L(t)$ guarantee that the solution $\psi$ converges to zero, while other conditions guarantee that $\psi(x,t)\geq B\sin\left(\frac{\pi x}{L(t)}\right)$ for some $B>0$. Even if $L(t)$ is strictly less than $L_{crit}$ for all $t$, then either outcome may occur. In particular, if $L(t)$ tends to $L_{crit}$ from below as some inverse power of $t$, then the outcome depends on this power.

The above reaction-diffusion equation has applications to population dynamics, and the domain with moving boundaries is relevant for modelling a habitat whose size and location may change over time. This could be due to factors such as flooding, loss of snow cover, habitat destruction, forest fire, or the extension of range boundaries when previously unsuitable regions become suitable. With respect to this application, the results of the current paper signify that even if the habitat is maintained below the `critical' size, population extinction is not guaranteed --- the solution may have a non-trivial lower bound. This is not so on a fixed domain, but is a possibility here because of the time-dependence of the domain  boundaries. Our work is also relevant for any other application that can be modelled by a diffusion equation within a spatial domain whose boundary moves under some external influence.

We emphasise that this is not the same as a free boundary problem, in which the moving boundary would be determined as part of the solution.
A number of authors, beginning with Du and Lin \cite{DuLin}, have studied a free boundary problem with a KPP nonlinearity. They consider problems of the type
\begin{equation}\label{eq_fb1}
\frac{\partial u}{\partial t} = d \frac{\partial^2 u}{\partial x^2}+ u(a-bu)  \qquad \textrm{in } g(t)< x< h(t)
\end{equation}
\begin{equation}\label{eq_fb2}
u(g(t),t)=u(h(t),t)=0
\end{equation}
\begin{equation}\label{eq_freeboundary}
g'(t)=-\mu\frac{\partial u}{\partial x}(g(t),t) \qquad h'(t)=-\mu\frac{\partial u}{\partial x}(h(t),t) 
\end{equation}
where $d$, $a$, $b$, and $\mu$ are given positive constants. In their model equation $g(t)$ and $h(t)$ are determined as part of the solution, via equation \eqref{eq_freeboundary}, and the domain length $h(t)-g(t)$ is monotonically increasing in time. See \cite{DuLin}, \cite{DuLin2}, and \cite{BunDuKra}, where Du, Lin, Bunting and Krakowski prove a spreading/vanishing dichotomy: either $g(t)\rightarrow -\infty$, $h(t)\rightarrow +\infty$ and $u$ spreads at an asymptotically constant speed in both directions, or else $g(t)\rightarrow g_{\infty}$, $h(t)\rightarrow h_{\infty}$ with $h_{\infty}-g_{\infty} \leq \pi\sqrt{d/a}$ and there is `vanishing', i.e. $u\rightarrow 0$. 
In the vanishing case, we note that the domain is always enclosed by a stationary interval of the critical length, $\pi\sqrt{d/a}$. 
In \cite{DuLou}, Du and Lou extend the spreading/vanishing results to other monostable reaction terms $f(u)$, showing that the results hold quite generally and are not specific to the special case $f(u)=u(a-bu)$.
(In particular, and in contrast to section \ref{section_nonlinear} of our work, \cite{DuLou} does not need to deal separately with the cases where $f$ is or is not linear on some neighbourhood $[0,k_0)$ of $0$.)

It is important to be aware that the `vanishing' behaviour proven in the free boundary papers \cite{DuLin}, \cite{DuLou} is a consequence of the Stefan boundary conditions \eqref{eq_freeboundary} in their model, and not simply a consequence of the domain length being less than the critical length. In particular, these boundary conditions imply that \emph{if} $h'(t)$ and $g'(t)$ tend to zero then so does the gradient $\frac{\partial u}{\partial x}$ at the boundary, and so the solution cannot be bounded below by any positive multiple of $\sin\left(\frac{\pi x}{h_{\infty}-g_{\infty}}\right)$.
This is in contrast to our current paper (which is not a free boundary problem and does not impose the Stefan conditions \eqref{eq_freeboundary}, but rather prescribes the movement of the domain boundaries). We demonstrate cases in which the domain length satisfies $L(t)<L_{crit}$, $L(t)\rightarrow L_{crit}$, $\dot{L}(t)\rightarrow 0$, and the solution \emph{is} bounded below by a positive multiple of $\sin\left(\frac{\pi x}{L(t)}\right)$.

Many subsequent papers on the free boundary problem \eqref{eq_fb1}, \eqref{eq_fb2}, \eqref{eq_freeboundary} have investigated the case of spreading in great detail, but much less attention has been given to the case of vanishing. Indeed, to our knowledge, aside from \cite{DuLin}, \cite{DuLou} who considered the nonlinear free boundary problem and proved vanishing, there have been no other publications looking at a reaction-diffusion equation on a time-dependent domain of less than the critical length.

The recent paper \cite{JA1} by the current author considers equations \eqref{eq_psi}, \eqref{eq_psi_BC} on a time-dependent interval, and presents several exact solutions for the linear case. These involve domains where
$L(t)$ has the form $L(t)=\sqrt{at^2+2bt+L_0^2}$. As such $L(t)$ is either constant in time, or $L(t)\rightarrow 0$ in a finite time, or $L(t)\rightarrow\infty$ as $t\rightarrow\infty$. The other result of \cite{JA1} concerns the long-time behaviour near the boundaries, when the length of the interval tends to infinity with its endpoints moving at $\pm(2\sqrt{Df'(0)}t - \alpha\log(t+1)-\eta(t))$ with $\alpha>0$ and $\eta(t)=O(1)$. The approach in \cite{JA1} is based upon changes of variables (both of the spatial variable $x$ and the independent variable $\psi$) which allows exact solutions, subsolutions and supersolutions to be constructed for the transformed equation. This is the method which will also be employed in the current study, but here we use different upper and lower bounds on the transformed equation to derive a new sub- and supersolution pair.

The main result of this paper is contained in Theorem \ref{Theorem_comparison}, which yields sub- and supersolutions for the linear equation with general $A(t)$ and $L(t)$. The results in Section \ref{section_consequences} are then the applications of Theorem \ref{Theorem_comparison} to specific cases of interest. In particular, Corollary \ref{Corollary_nonconv} gives conditions on $L(t)$ under which the solution does not converge to zero in long time, and Corollary \ref{Corollary_0conv} gives conditions on $L(t)$ for which the solution does converge to zero. Examples of both are given. In Example \ref{example2} we consider $L(t)$ tending towards $L_{crit}$ from below like some inverse power of $t$, and find that the outcome is determined by the power of $t$. Thus we provide both a set of conditions on $L(t)$, and a family of examples, such that $L(t)$ is always strictly less than the `critical length' and yet the lower bound $\psi(x,t)\geq B\sin\left(\frac{\pi x}{L(t)}\right)$ is valid for some constant $B>0$. We finish the paper by considering the nonlinear equation, in Section \ref{section_nonlinear}.

\section{Derivation of sub- and supersolutions}
We consider the linear equation and begin as in \cite{JA1} by transforming
onto a fixed spatial domain and making a change of variables. Let $\psi(x,t)=u(\xi,t)$ where $\xi=\frac{(x-A(t))}{L(t)}L_0$ (for some $L_0>0$) to get
\begin{equation}\label{eq_u}
\frac{\partial u}{\partial t} = D \frac{L_0^2 }{L(t)^2} \frac{\partial^2 u}{\partial \xi^2}+\left(\frac{\dot{A}( t)L_0+\xi\dot{L}(t)}{L(t)}\right)\frac{\partial u}{\partial \xi} +f'(0) u  \qquad \textrm{in } 0< \xi< L_0
\end{equation}
\begin{equation}
u(\xi,t)=0 \qquad \textrm\qquad\textrm{at } \xi=0 \textrm{ and } \xi=L_0.
\end{equation}
Then let
\begin{align}\label{eq_wdef}
w(\xi, t)=& u(\xi, t)\left(\frac{L(t)}{L_0}\right)^{1/2}\exp{\left(-f'(0)t+\int\limits_0^ t\frac{\dot{A}(\zeta)^2 }{4D}d\zeta\right)}\nonumber\\
&\times\exp{\left(\frac{\xi^2 \dot{L}(t)L(t)}{4DL_0^2}+\frac{\xi\dot{A}(t)L(t)}{2DL_0}\right)}
\end{align}
and find that $w$ satisfies
\begin{equation}\label{eq_weq}
\frac{\partial w}{\partial t} = D \frac{L_0^2 }{L(t)^2} \frac{\partial^2 w}{\partial \xi^2}+\left( \frac{\xi^2\ddot{L}(t)L(t)}{4DL_0^2} + \frac{\xi\ddot{A}(t)L(t)}{2DL_0} \right)w  \qquad \textrm{in } 0< \xi< L_0
\end{equation}
\begin{equation}
w(\xi,t)=0 \qquad \textrm\qquad\textrm{at } \xi=0 \textrm{ and } \xi=L_0.
\end{equation}
Now define the non-negative functions $\overline{Q}(t)$ and $\underline{Q}(t)$ as follows:
\begin{equation}
\overline{Q}(t)=\max_{0\leq\eta\leq 1} \left(\frac{\eta^2\ddot{L}(t)L(t)}{2} + \eta \ddot{A}(t)L(t) \right),
\end{equation}
\begin{equation}
\underline{Q}(t)=-\min_{0\leq\eta\leq 1} \left( \frac{\eta^2\ddot{L}(t)L(t)}{2} + \eta \ddot{A}(t)L(t) \right).
\end{equation}

\begin{theorem} \label{Theorem_comparison}
Let $\psi(x,t)$ satisfy the linear problem on $A(t)<x<A(t)+L(t)$, and let $u(\xi,t)$ and $w(\xi,t)$ be as above. Suppose that $b \sin\left(\frac{\pi\xi}{L_0}\right) \leq w(\xi,0) \leq a \sin\left(\frac{\pi\xi}{L_0}\right)$ for some $0<b\leq a$.
Then for all $t\geq 0$ we have
\begin{align}
u(\xi,t) \geq & b \sin\left(\frac{\pi\xi}{L_0}\right)\left(\frac{L_0}{L(t)}\right)^{1/2}\exp{\left(f'(0)t -\int_0^t \frac{\dot{A}(\zeta)^2}{4D}d\zeta \right)}\nonumber\\
&\times\exp{\left(\int_0^t \left(-\frac{D\pi^2}{L(\zeta)^2}-\frac{\underline{Q}(\zeta)}{2D}\right) d\zeta -\frac{\xi^2\dot{L}(t)L(t)}{4DL_0^2} - \frac{\xi\dot{A}(t)L(t)}{2DL_0}\right)}
\end{align}
and
\begin{align}
u(\xi,t)\leq & a \sin\left(\frac{\pi\xi}{L_0}\right)\left(\frac{L_0}{L(t)}\right)^{1/2}\exp{\left(f'(0)t -\int_0^t \frac{\dot{A}(\zeta)^2}{4D}d\zeta \right)}\nonumber\\
&\times\exp{\left(\int_0^t \left(-\frac{D\pi^2}{L(\zeta)^2}+\frac{\overline{Q}(\zeta)}{2D}\right) d\zeta -\frac{\xi^2\dot{L}(t)L(t)}{4DL_0^2} - \frac{\xi\dot{A}(t)L(t)}{2DL_0}\right)}.
\end{align}
\end{theorem}
\begin{proof}
By the definitions of $\overline{Q}(t)$ and $\underline{Q}(t)$, and the equation \eqref{eq_weq} satisfied by $w(\xi,t)\geq0$, we have
\begin{equation}
-\frac{\underline{Q}(t)}{2D}w(\xi,t)\leq \frac{\partial w}{\partial t}-D\frac{L_0^2}{L(t)^2}\frac{\partial^2 w}{\partial \xi^2}  \leq \frac{\overline{Q}(t)}{2D}w(\xi,t)
\end{equation}
uniformly in $0\leq\xi\leq L_0$.
This then leads to the sub- and supersolutions
\begin{equation}
b \sin\left(\frac{\pi\xi}{L_0}\right)\exp{\left(\int_0^t \left(-\frac{D\pi^2}{L(\zeta)^2}-\frac{\underline{Q}(\zeta)}{2D}\right)d\zeta\right)}\leq w(\xi,t),
\end{equation}
\begin{equation}
w(\xi,t)\leq a \sin\left(\frac{\pi\xi}{L_0}\right)\exp{\left(\int_0^t \left(-\frac{D\pi^2}{L(\zeta)^2}+\frac{\overline{Q}(\zeta)}{2D}\right)d\zeta\right)}.
\end{equation}
Changing variables back using equation \eqref{eq_wdef} gives the claimed bounds for $u(\xi,t)$.
\end{proof}

\section{Non-convergence to zero and convergence to zero}\label{section_consequences}
In this section we continue to consider the linear reaction term, and we prove conditions on $L(t)$ under which the solution does not, or does, tend to zero as $t\rightarrow\infty$.
For any function $F(t)$, denote its positive and negative parts by $[F(t)]^{+}\geq0$ and $[F(t)]^{-}\geq0$, so that $F(t)\equiv [F(t)]^{+} -[F(t)]^{-}$.

\begin{corollary}\label{Corollary_nonconv} Non-convergence to zero.\\
Let $\psi(x,t)$ satisfy the linear problem on $0<x<L(t)$, and assume
that the following conditions hold:
\begin{equation}\label{eq_cond_1}
L(t) \textrm{ and }\dot{L}(t)L(t) \textrm{ are bounded above},
\end{equation}
\begin{equation}\label{eq_cond_2}
\int_0^t \left(\frac{1}{L(\zeta)^2}- \frac{1}{L_{crit}^2}\right) d\zeta \textrm{ is bounded above},
\end{equation}
\begin{equation}\label{eq_cond_3}
\int_0^{\infty} L(t)[\ddot{L}(t)]^{-} dt <\infty.
\end{equation}
Then for $\psi(x,0)\geq 0$, and not identically zero, $\psi(x,t)$ does not converge to zero as $t\rightarrow \infty$. In particular
\begin{equation}\label{eq_Fourier_coeff}
\liminf_{t\rightarrow\infty}\left(\frac{2}{L(t)}\int_0^{L(t)} \psi(x,t)\sin\left(\frac{\pi x}{L(t)}\right)dx\right) >0.
\end{equation}
\end{corollary}
\begin{proof}
We may assume without loss of generality that
\begin{equation}
u(\xi,0)\geq b \sin\left(\frac{\pi\xi}{L_0}\right)\exp{\left( -\frac{\xi^2\dot{L}(0)L(0)}{4DL_0^2} \right)}
\end{equation}
for some $b>0$, since by the strong parabolic maximum principle and Hopf's Lemma  such an inequality will be true at each time $t_0>0$. It follows from Theorem \ref{Theorem_comparison} that
\begin{align}
u(\xi,t)\geq & b \sin\left(\frac{\pi\xi}{L_0}\right)\left(\frac{L_0}{L(t)}\right)^{1/2}\exp{(f'(0)t)}\nonumber\\
&\times\exp{\left(\int_0^t \left(-\frac{D\pi^2}{L(\zeta)^2}-\frac{\underline{Q}(\zeta)}{2D}\right) d\zeta -\frac{\xi^2\dot{L}(t)L(t)}{4DL_0^2} \right)}
\end{align}
for all $t\geq 0$, where $\underline{Q}(t)=\frac{L(t)[\ddot{L}(t)]^{-}}{2}$. Substitute $f'(0)=\frac{D\pi^2}{L_{crit}^2}$, and use the assumptions in equation \eqref{eq_cond_1}, to get that for some $b'>0$
\begin{equation}
u(\xi,t)\geq b' \sin\left(\frac{\pi\xi}{L_0}\right)\exp{\left(\int_0^t \left(\frac{D\pi^2}{L_{crit}^2}-\frac{D\pi^2}{L(\zeta)^2}-\frac{L(\zeta)[\ddot{L}(\zeta)]^{-}}{4D}\right) d\zeta \right)}.
\end{equation}
Now the assumptions in equations \eqref{eq_cond_2} and \eqref{eq_cond_3} give that this exponential factor is bounded below by a positive value: there exists $B>0$ such that 
\begin{equation}
u(\xi,t)\geq B \sin\left(\frac{\pi\xi}{L_0}\right) \qquad \textrm{for all } t\geq 0.
\end{equation}
So the first Fourier sine coefficient is bounded below as
\begin{equation}
\frac{2}{L(t)}\int_0^{L(t)} \psi(x,t)\sin\left(\frac{\pi x}{L(t)}\right)dx = \frac{2}{L_0}\int_0^{L_0} u(\xi,t)\sin\left(\frac{\pi \xi}{L_0}\right)d\xi \geq B
\end{equation}
and thus we reach the conclusion.
\end{proof}

\begin{corollary}\label{Corollary_0conv}  Convergence to zero.\\
Let $\psi(x,t)$ satisfy the linear problem on $0<x<L(t)$, and assume
that the following conditions hold:
\begin{equation}\label{eq_cond_4}
\dot{L}(t)L(t) \textrm{ is bounded below},
\end{equation}
\begin{equation}\label{eq_cond_5}
\int_0^T  \left(\frac{D\pi^2}{L(t)^2}- \frac{D\pi^2}{L_{crit}^2} -\frac{L(t)[\ddot{L}(t)]^{+}}{4D} + \frac{\dot{L}(t)}{2L(t)} \right) dt \rightarrow \infty \qquad \textrm{as }T\rightarrow\infty.
\end{equation}
Then $\psi(x,t)$ converges uniformly to zero as $t\rightarrow \infty$.
\end{corollary}
\begin{proof}
We may assume without loss of generality that
\begin{equation}
u(\xi,0)\leq a \sin\left(\frac{\pi\xi}{L_0}\right)\exp{\left( -\frac{\xi^2\dot{L}(0)L(0)}{4DL_0^2} \right)}
\end{equation}
for some $a>0$. It follows from Theorem \ref{Theorem_comparison} that
\begin{align}
u(\xi,t)\leq & \ a \sin\left(\frac{\pi\xi}{L_0}\right)\left(\frac{L_0}{L(t)}\right)^{1/2}\exp{(f'(0)t)}\nonumber\\
&\times\exp{\left(\int_0^t \left(-\frac{D\pi^2}{L(\zeta)^2}+\frac{\overline{Q}(\zeta)}{2D}\right) d\zeta -\frac{\xi^2\dot{L}(t)L(t)}{4DL_0^2} \right)}
\end{align}
for all $t\geq 0$, where $\overline{Q}(t)=\frac{L(t)[\ddot{L}(t)]^{+}}{2}$. Substitute $f'(0)=\frac{D\pi^2}{L_{crit}^2}$ and use the assumption in equation \eqref{eq_cond_4}, to get that for some $a'>0$
\begin{equation}
u(\xi,t)\leq a' \sin\left(\frac{\pi\xi}{L_0}\right)\exp{\left(\int_0^t \left(\frac{D\pi^2}{L_{crit}^2}-\frac{D\pi^2}{L(\zeta)^2}+\frac{L(\zeta)[\ddot{L}(\zeta)]^{+}}{4D} - \frac{\dot{L}(\zeta)}{2L(\zeta)}\right) d\zeta \right)}.
\end{equation}
Thus, under the assumption in equation \eqref{eq_cond_5}, there is uniform convergence to zero.
\end{proof}

\begin{example}\label{example1}
Consider an interval $(0,L(t))$ where $L(t)<L_{crit}$ tends exponentially towards $L_{crit}$:
\begin{equation}
L(t)=L_{crit}(1-\varepsilon e^{-\alpha t})
\end{equation}
for some $0<\varepsilon<1$ and $\alpha>0$.
The conditions of Corollary \ref{Corollary_nonconv} are satisfied and so we deduce that $\psi(x,t)\geq B\sin\left(\frac{\pi x}{L(t)}\right)$ for some $B>0$.
\end{example}

\begin{example}\label{example2}
Consider the interval $(0,L(t))$ where $L(t)<L_{crit}$ is given by
\begin{equation}
L(t)=L_{crit}(1-\varepsilon (t+1)^{-k})
\end{equation}
for some $0<\varepsilon<1$ and $k>0$.
If $k>1$, then the conditions of Corollary \ref{Corollary_nonconv} are satisfied  and thus in this case $\psi(x,t)\geq B\sin\left(\frac{\pi x}{L(t)}\right)$ for some $B>0$. On the other hand, if $0<k\leq 1$,
the conditions of Corollary \ref{Corollary_0conv} are satisfied and so $\psi(x,t)\rightarrow 0$ uniformly in $x$ as $t\rightarrow\infty$. This example gives an indication of how fast or slowly $L(t)$ may be expected to converge to $L_{crit}$, to give each of the two outcomes.
\end{example}

Analogous results to Corollaries \ref{Corollary_nonconv} and \ref{Corollary_0conv} also hold in the case of an interval $(A_0+ct, A_0+ct+L(t))$ where $c$ is a constant with $\vert c \vert<2\sqrt{Df'(0)}$. In this case we define
\begin{equation}\label{eq_c_Lcrit}
L_{crit}(c)=\pi\sqrt{\frac{D}{f'(0)-\frac{c^2}{4D}}}.
\end{equation}
and, exactly as above but with this new definition of $L_{crit}(c)$, we obtain the following corollaries of Theorem \ref{Theorem_comparison}.
\begin{corollary}
Let $\psi(x,t)$ satisfy the linear problem on $(A_0+ct, A_0+ct+L(t))$, with $\vert c \vert<2\sqrt{Df'(0)}$, and assume that the following conditions hold:
\begin{equation}
L(t), \textrm{ } \dot{L}(t)L(t) \textrm{ and } cL(t) \textrm{ are bounded above},
\end{equation}
\begin{equation}
\int_0^t \left(\frac{1}{L(\zeta)^2}- \frac{1}{L_{crit}(c)^2}\right) d\zeta \textrm{ is bounded above},
\end{equation}
\begin{equation}
\int_0^{\infty} L(t)[\ddot{L}(t)]^{-} dt <\infty.
\end{equation}
Then for $\psi(x,0)\geq 0$, and not identically zero, $\psi(x,t)$ does not converge to zero as $t\rightarrow \infty$.
\end{corollary}
\begin{corollary}
Let $\psi(x,t)$ satisfy the linear problem on $(A_0+ct, A_0+ct+L(t))$, with $\vert c \vert<2\sqrt{Df'(0)}$, and assume that the following conditions hold:
\begin{equation}
\dot{L}(t)L(t) \textrm{ and } cL(t) \textrm{ are  bounded below},
\end{equation}
\begin{equation}
\int_0^T  \left(\frac{D\pi^2}{L(t)^2}- \frac{D\pi^2}{L_{crit}(c)^2} -\frac{L(t)[\ddot{L}(t)]^{+}}{4D} + \frac{\dot{L}(t)}{2L(t)} \right) dt \rightarrow \infty \qquad \textrm{as }T\rightarrow\infty.
\end{equation}
Then $\psi(x,t)$ converges uniformly to zero as $t\rightarrow \infty$.
\end{corollary}

\section{Nonlinear reaction terms of KPP type}\label{section_nonlinear}
Consider the nonlinear version of the equation \eqref{eq_psi}, where $f$ is a function of KPP type. Note that the linear solution is now a supersolution. Consequently, Corollary \ref{Corollary_0conv} immediately extends to the nonlinear case, and we obtain the following.
\begin{corollary}Convergence to zero.\\
Let $\psi(x,t)$ satisfy the nonlinear problem on $0<x<L(t)$, where $f$ satisfies \eqref{eq_f}. Assume
that \eqref{eq_cond_4} and \eqref{eq_cond_5} hold.
Then $\psi(x,t)$ converges uniformly to zero as $t\rightarrow \infty$.
\end{corollary}

The results concerning the lower bound now depend on whether or not $f$ is linear in some neighbourhood of the origin.
\begin{proposition} Non-convergence to zero.\\
Let $\psi(x,t)$ satisfy the nonlinear problem on $0<x<L(t)$, where $f$ satisfies \eqref{eq_f}, and where there is some $k_0>0$ such that 
\begin{equation}
f(k)=f'(0)k \textrm{ for } 0\leq k\leq k_0.
\end{equation}
Assume that there exist finite constants $m_1$, $m_2$, $M$, $I_1$, $I_2$ such that for all $t\geq 0$
\begin{equation}
0<m_1\leq L(t)\leq m_2
\end{equation}
\begin{equation}
\vert\dot{L}(t)L(t)\vert \leq M,
\end{equation}
\begin{equation}
\left\vert\int_0^t \left(\frac{1}{L(\zeta)^2}- \frac{1}{L_{crit}^2}\right) d\zeta \right\vert \leq I_1,
\end{equation}
\begin{equation}
\int_0^{t} L(\zeta)[\ddot{L}(\zeta)]^{-} d\zeta \leq I_2.
\end{equation}
Then for $\psi(x,0)\geq 0$, and not identically zero, $\psi(x,t)$ does not converge to zero as $t\rightarrow \infty$. In particular, equation \eqref{eq_Fourier_coeff} holds.
\end{proposition}
\begin{proof}
Choose $\hat{b}>0$ small enough such that both
\begin{equation}
u(\xi,0)\geq \hat{b} \sin\left(\frac{\pi\xi}{L_0}\right)\exp{\left( -\frac{\xi^2\dot{L}(0)L(0)}{4DL_0^2} \right)}
\end{equation}
and
\begin{equation}\label{eq_b_k0}
\hat{b}\left(\frac{L_0}{m_1}\right)^{1/2}\exp\left(I_1+ \frac{M}{4D}\right) \leq k_0.
\end{equation}
By Theorem \ref{Theorem_comparison}, the function
\begin{align}
\hat{u}(\xi,t)=& \hat{b} \sin\left(\frac{\pi\xi}{L_0}\right)\left(\frac{L_0}{L(t)}\right)^{1/2}\exp{(f'(0)t)} \nonumber\\
& \times \exp{\left(\int_0^t \left( -\frac{D\pi^2}{L(\zeta)^2}-\frac{\underline{Q}(\zeta)}{2D}\right) d\zeta -\frac{\xi^2\dot{L}(t)L(t)}{4DL_0^2} \right)}
\end{align}
is a subsolution for the linear equation, where $\underline{Q}(t)=\frac{L(t)[\ddot{L}(t)]^{-}}{2}$. By choice of $\hat{b}$, we have $0\leq \hat{u}(\xi,t) \leq k_0$ for all $0\leq t < \infty$, and hence $\hat{u}$ is also a subsolution for the nonlinear equation.
Therefore
\begin{equation}\label{eq_u_lowerbound}
u(\xi,t)\geq \hat{u}(\xi,t) \geq \hat{b}\left(\frac{L_0}{m_2}\right)^{1/2}\exp\left(-I_1 -\frac{I_2}{4D} - \frac{M}{4D}\right) \sin\left(\frac{\pi\xi}{L_0}\right)
\end{equation}
for all $t\geq 0$ and the conclusion follows.
\end{proof}
\begin{remark}
The lower bound for $u(\xi,t)$ in equation \eqref{eq_u_lowerbound} depends on the parameter $k_0$ via equation \eqref{eq_b_k0}. As $k_0\rightarrow 0$, $\hat{b}=O(k_0)$ and so the lower bound in equation \eqref{eq_u_lowerbound} is also $O(k_0)$. This is in keeping with the next result, which considers functions $f$ which are not linear on any neighbourhood of $0$, and no such lower bound exists as $L(t)\rightarrow L_{crit}$.
\end{remark}
\begin{proposition}Convergence to zero.\\
Let $\psi(x,t)$ satisfy the nonlinear problem on $0<x<L(t)$, where $f$ satisfies \eqref{eq_f}, and where there is some $k_0>0$ such that 
\begin{equation}\label{eq_f_strict}
f(k)<f'(0)k \textrm{ for } 0< k< k_0.
\end{equation}
If $\limsup_{t\rightarrow\infty} L(t)\leq L_{crit}$ then $\psi(x,t)$ converges uniformly to zero as $t\rightarrow \infty$.
\end{proposition}
\begin{proof}
For each $\varepsilon>0$, let $L_{\varepsilon}=L_{crit}(1+\varepsilon)$ and consider the solution $\psi_{\varepsilon}(x,t)$ to the nonlinear parabolic equation on the fixed interval $(0, L_{\varepsilon})$. Standard arguments show that as $t\rightarrow\infty$, $\psi_{\varepsilon}$ converges uniformly to a non-negative solution of the elliptic equation
\begin{equation}\label{eq_elliptic}
DU_{\varepsilon}''(x)+f(U_{\varepsilon}(x)) = 0 \qquad \textrm{in } 0< x< L_{\varepsilon}
\end{equation}
\begin{equation}\label{eq_ellipticBC}
U_{\varepsilon}(0)=U_{\varepsilon}(L_{\varepsilon})=0.
\end{equation}
Since $L_{\varepsilon}>L_{crit}$, and $f$ satisfies \eqref{eq_f}, a positive solution $U_{\varepsilon}$ to equations \eqref{eq_elliptic}, \eqref{eq_ellipticBC} is unique (see Theorem 1 and Remark 1 of \cite{BreOsw}).
Now for $t$ large enough, $L(t)\leq L_{\varepsilon}$ and $\psi_{\varepsilon}(x,t)$ is a supersolution for $\psi(x,t)$. Therefore it must hold that
\begin{equation}\label{eq_limsup}
\limsup_{t\rightarrow \infty} \left(\sup_x \psi(x,t)\right) \leq \vert\vert U_{\varepsilon} \vert\vert_{\infty} \qquad \textrm{for each } \varepsilon>0.
\end{equation}
But as $\varepsilon\rightarrow 0$ the functions $U_{\varepsilon}$ converge uniformly to a non-negative solution $U$ of
equations \eqref{eq_elliptic}, \eqref{eq_ellipticBC} on the interval of length $L_0=L_{crit}$.
This limit $U$ therefore satisfies
\begin{align}\label{eq_intU}
\int_0^{L_{crit}}DU'(x)^2dx = \int_0^{L_{crit}}-DU(x)U''(x)dx &= \int_0^{L_{crit}} f(U(x))U(x)dx \nonumber\\
&\leq \int_0^{L_{crit}} f'(0)U(x)^2dx,
\end{align}
where the inequality follows from $f(k)\leq f'(0)k$.
Now, since $f(k)<f'(0)k$ on $(0,k_0)$, and since $U(x)=0$ at the endpoints, the inequality in \eqref{eq_intU} would be strict unless $U\equiv 0$.
But $f'(0)=\frac{D\pi^2}{L_{crit}^2}$, and so a strict inequality would contradict Poincar\'{e}'s inequality.
Therefore we deduce that indeed $U(x)\equiv 0$ on $[0,L_{crit}]$.
So, $\vert\vert U_{\varepsilon} \vert\vert_{\infty} \rightarrow \vert\vert U \vert\vert_{\infty}=0$ as $\varepsilon\rightarrow 0$, and it follows from \eqref{eq_limsup} that $\lim_{t\rightarrow \infty} \left(\sup_x \psi(x,t)\right) =0$.
\end{proof}

\section*{Acknowledgements}
I am very grateful to my PhD supervisor, Professor Elaine Crooks. I am also grateful for an EPSRC-funded studentship (EPSRC DTP grant EP/R51312X/1).

\end{document}